\documentclass{amsart}
\usepackage{mathtools, microtype, mathrsfs, xfrac, hyperref, amssymb, enumerate, xspace, stmaryrd}
\usepackage[latin1]{inputenc}
\usepackage{tikz-cd}
\usepackage{todonotes}

\numberwithin{equation}{section}

\numberwithin{subsection}{section}

\allowdisplaybreaks[1]

\newenvironment{enumerate1}
{\begin{enumerate}[\upshape (1)]}
{\end{enumerate}}

\newtheorem*{namedtheorem}{\theoremname}
\newcommand{\theoremname}{testing}

\newtheorem{theorem}{Theorem}[section]
\newtheorem{proposition}[theorem]{Proposition}
\newtheorem{proposition-definition}[theorem]
{Proposition-Definition}
\newtheorem{corollary}[theorem]{Corollary}
\newtheorem{lemma}[theorem]{Lemma}

\theoremstyle{definition}
\newtheorem{definition}[theorem]{Definition}

\newtheorem{example}[theorem]{Example}

\theoremstyle{remark}

\renewcommand{\mathcal}{\mathscr}

 \newcommand\cB{\mathcal{B}}

\newcommand\cM{\mathcal{M}} 
\newcommand\cO{\mathcal{O}}

\newcommand\GG{\mathbb{G}}

 \newcommand\ZZ{\mathbb{Z}}

 \newcommand\frp{\mathfrak{p}}

\newcommand\arr{\ifinner\to\else\longrightarrow\fi}

\newcommand\arrto{\ifinner\mapsto\else\longmapsto\fi}

\newcommand{\eqdef}{\mathrel{\smash{\overset{\mathrm{\scriptscriptstyle def}} =}}}

\renewcommand\th{^\text{th}}

\def\displaytimes_#1{\mathrel{\mathop{\times}\limits_{#1}}}

\def\displayotimes_#1{\mathrel{\mathop{\bigotimes}\limits_{#1}}}

\newcommand\spec{\operatorname{Spec}}

\newcommand{\underisom}{\mathop{\underline{\mathrm{Isom}}}\nolimits}

\newlength{\ignora}

\newcommand{\mmu}{\boldsymbol{\mu}}

\newcommand\radice[2][\relax]{\hspace{-1.5pt}\sqrt[\uproot{2}#1]{#2}}

\DeclareFontFamily{U}{mathx}{\hyphenchar\font45}
\DeclareFontShape{U}{mathx}{m}{n}{
      <5> <6> <7> <8> <9> <10>
      <10.95> <12> <14.4> <17.28> <20.74> <24.88>
      mathx10
      }{}
\DeclareSymbolFont{mathx}{U}{mathx}{m}{n}
\DeclareFontSubstitution{U}{mathx}{m}{n}
\DeclareMathAccent{\widecheck}{0}{mathx}{"71}
\DeclareMathAccent{\wideparen}{0}{mathx}{"75}

\renewcommand{\epsilon}{\varepsilon}

\begin{document}

\title[An arithmetic valuative criterion for tame algebraic stacks]{An arithmetic valuative criterion\\for proper maps of tame algebraic stacks}

\author[Bresciani]{Giulio Bresciani}
\author[Vistoli]{Angelo Vistoli}

\address{Scuola Normale Superiore\\Piazza dei Cavalieri 7\\
56126 Pisa\\ Italy}
\email[Vistoli]{angelo.vistoli@sns.it}
\email[Bresciani]{giulio.bresciani@gmail.com}

\thanks{The second author was partially supported by research funds from Scuola Normale Superiore, project \texttt{SNS19\_B\_VISTOLI}, and by PRIN project ``Derived and underived algebraic stacks and applications''.  The paper is based upon work partially supported by the Swedish Research Council under grant no. 2016-06596 while the second author was in residence at Institut Mittag-Leffler in Djursholm}

%\date{\arrday}

\begin{abstract}
	The valuative criterion for proper maps of schemes has many applications in arithmetic, e.g. specializing $\mathbb{Q}_{p}$-points to $\mathbb{F}_{p}$-points. For algebraic stacks, the usual valuative criterion for proper maps is ill-suited for these kind of arguments, since it only gives a specialization point defined over an extension of the residue field, e.g. a $\mathbb{Q}_{p}$-point will specialize to an $\mathbb{F}_{p^{n}}$-point for some $n$. We give a new valuative criterion for proper maps of tame stacks which solves this problem and is well-suited for arithmetic applications. As a consequence, we prove that the Lang-Nishimura theorem holds for tame stacks.
\end{abstract}

\maketitle

\section{Introduction}

The well known and extremely useful valuative criterion for properness says, in particular, that if $X \arr Y$ is a proper morphism of schemes, $R$ is a DVR with quotient field $K$ and residue field $k$, and we have a commutative diagram
   \[
   \begin{tikzcd}
   \spec K \ar[d, hook]\rar & X\dar\\
   \spec R \rar & S
   \end{tikzcd}
   \]
there exists a unique lifting $\spec R \arr X$ of $\spec R \arr Y$ extending $\spec K \arr X$. This has many arithmetic applications: most of them use that the statement above ensures the existence of a lifting $\spec k \arr X$ of the composite $\spec k \subseteq \spec R \arr Y$.

If $X$ and $Y$ are algebraic stack, and $X \arr Y$ is a proper morphism, then this fails, even in very simple examples, unless $X \arr Y$ is representable. The correct general statement is that there exists a local extension of DVR $R \arr R'$, such that if we denote by $K'$ the fraction field of $R'$, the composite $\spec R' \arr \spec R \arr Y$ has a lifting $\spec R' \arr X$ extending the composite $\spec K' \arr \spec K \arr X$ (see for example  \cite[\href{https://stacks.math.columbia.edu/tag/0CLY}{Tag 0CLY}]{stacks-project}). For arithmetic applications this is problematic, because the extension $R \subseteq R'$ will typically induce a nontrivial extension of residue fields, so it does not imply that $\spec k \arr Y$ lifts to $\spec k \arr X$, as in the case of schemes.

When $X$ and $Y$ are Deligne--Mumford stacks over a field of characteristic~$0$ a substitute was found by the first author in \cite[Theorem~1]{giulio-section-birational}. In this note we we extend this, in a somewhat more precise form, to positive and mixed characteristic. In this context the correct generality is that of \emph{tame stacks}, in the sense of \cite{dan-olsson-vistoli1}. Tame stacks are algebraic stacks with finite inertia, such that the automorphism group scheme of any object over a field is linearly reductive. In characteristic~$0$ they coincide with Deligne--Mumford stacks with finite inertia, but in positive and mixed characteristic there are Deligne--Mumford stacks with finite inertia that are not tame, and tame stacks that are not Deligne--Mumford.

In our version the role of the DVR $R'$ above is played by a \emph{root stack} $\radice[n]{\spec R}$; this is not a scheme, but a tame stack with a map $\radice[n]{\spec R} \arr \spec R$, which is an isomorphism above $\spec K \subseteq \spec R$ (see the discussion at the beginning of \S\ref{sec:valuative}). The statement of our main theorem \ref{thm:valuative} is that if $X \arr Y$ is a proper morphism of tame algebraic stacks and we have a commutative diagram as above, there exists a unique positive integer $n$ and a unique representable lifting $\radice[n]{\spec R} \arr X$ of the composite $\radice[n]{\spec R} \arr \spec R \arr Y$ extending $\spec K \arr X$. The key point for arithmetic applications is that the closed point $\spec k \arr \spec R$ lifts to $\spec k \arr \radice[n]{\spec R}$. This statement is much harder to prove in arbitrary characteristic than in characteristic~$0$.

Besides the original application to Grothendieck's section conjecture in \cite{giulio-section-birational}, this valuative criterion has been applied in \cite{giulio-angelo-genericity} to give new proofs and stronger versions of the genericity theorem for essential dimension. 

Recall that the Lang-Nishimura theorem states that the property of having a rational point is a birational invariant of smooth proper varieties. Another consequence of our version of the valuative criterion is that the Lang-Nishimura theorem generalizes to tame stacks, see \ref{LN}. Our version of the Lang-Nishimura theorem has an immediate corollary, which we find surprising: if $\cM$ is a smooth tame stack which is generically a scheme and $\overline{M}\to M$ is a resolution of singularities of the coarse moduli space $\cM\to M$, then a rational point of $M(k)$ lifts to $\cM$ if and only if it lifts to $\overline{M}$. This gives a hint of the applications of the Lang--Nishimura theorem to fields of moduli, which will be the subject of the forthcoming papers \cite{giulio-angelo-fields-moduli, giulio-singularities}.

\section{Notations and conventions}

We will follow the conventions of \cite{knutson} and \cite{laumon-moret-bailly}; so the diagonals of algebraic spaces and algebraic stacks will be separated and of finite type. In particular, every algebraic space will be \emph{decent}, in the sense of \cite[Definition~03I8]{stacks-project}.

We will follow the terminology of \cite{dan-olsson-vistoli1}: a \emph{tame stack} is an algebraic stack $X$ with finite inertia, such that its geometric points have linearly reductive automorphism group. This is equivalent to requiring that $X$ is étale locally over its moduli space a quotient by a finite, linearly reductive group scheme \cite[Theorem 3.2]{dan-olsson-vistoli1}.

More generally, a morphism $f\colon X \arr Y$ of algebraic stacks is \emph{tame} if the relative inertia group stack $I_{X/Y} \arr X$, defined as in \cite[\href{https://stacks.math.columbia.edu/tag/050P}{Section 050P}]{stacks-project}, is finite and has linearly reductive geometric fibers. See \cite[\S 3]{dan-olsson-vistoli-2}.

Using \cite[\href{https://stacks.math.columbia.edu/tag/0CPK}{Lemma 0CPK}]{stacks-project} one can easily prove the following.

\begin{proposition}
Let $f\colon X \arr Y$ be a morphism of algebraic stacks. The following conditions are equivalent.

\begin{enumerate1}

\item $f$ is tame.

\item If $Z \arr Y$ is a morphism, and $Z$ is a scheme, then $Z \times_{Y}X$ is a tame stack.

\item If $Z \arr Y$ is a morphism, and $Z$ a tame stack, then $Z \times_{Y}X$ is also tame.

\end{enumerate1}

Furthermore, if $X$ is tame, then the morphism $f$ is also tame.
\end{proposition}

\section{The valuative criterion}\label{sec:valuative}

A basic example of tame stacks is \emph{root stacks} (see \cite[Appendix~B2]{dan-tom-angelo2008}). We will need this in the following situation. Let $R$ be a DVR with uniformizing parameter $\pi$ and residue field $k \eqdef R/(\pi)$. If $n$ is a positive integer, we will denote by $\radice[n]{\spec R}$ the $n\th$ root of the Cartier divisor $\spec k\subseteq \spec R$. It is a stack over $\spec R$, such that given a morphism $\phi\colon T \arr \spec R$, the groupoid of liftings $T \arr \radice[n]{\spec R}$ is equivalent to the groupoid whose objects are triples $(L, s, \alpha)$, where $L$ is an invertible sheaf on $T$, $s \in L(T)$ is a global section of $L$, and $\alpha$ is an isomorphism $L^{\otimes n} \simeq \cO_{T}$, such that $\alpha(s^{\otimes n}) = \phi^\sharp(\pi)$. Alternatively, $\radice[n]{\spec R}$ can be described as the quotient stack $[\spec R[t]/(t^{n} - \pi)/\mmu_{n}]$, where the action of $\mmu_{n}$ on $\spec R[t]/(t^{n} - \pi)$ is by multiplication on $t$. The morphism $\rho\colon \radice[n]{\spec R} \arr \spec R$ is an isomorphism outside of $\spec k \subseteq \spec R$, while the reduced fiber $\rho^{-1}(\spec k)_{\mathrm{red}}$ is non-canonically isomorphic to the classifying stack $\cB_{k}\mmu_{n}$. In particular the embedding $\spec k \arr \spec R$ lifts to a morphism $\spec k \arr \radice[n]{\spec R}$.

The following is our version of the valuative criterion for properness.

\begin{theorem}\label{thm:valuative}
Let $f\colon X \arr Y$ be a tame, proper morphism of algebraic stacks, $R$ a DVR with quotient field $K$. Suppose that we have a $2$-commutative square
   \[
   \begin{tikzcd}
   \spec K \ar[d, hook]\rar & X \dar{f}\\
   \spec R \rar & Y
   \end{tikzcd}
   \]
Then there exists a unique positive integer $n$ and a representable lifting $\radice[n]{\spec R} \arr X$ of the given morphism $\spec R \arr Y$, making the diagram
	\[\begin{tikzcd}
    &	\spec K\rar\ar[dl, hook]	\	&	X\dar	& 	\\
		\radice[n]{\spec R}\rar\ar[rru]	&	\spec R\rar\ar[from=u,hook,crossing over]	&	Y		& 
	\end{tikzcd}\]
$2$-commutative. Furthermore, the lifting is unique up to a unique isomorphism.
\end{theorem}

\begin{corollary}\label{cor:valuative}
In the situation above, if $k$ is the residue field of $R$, the composite $\spec k \subseteq \spec R \arr Y$ has a lifting $\spec k \arr X$.
\end{corollary}

As one would expect, these statements fail without the tameness hypothesis, even when $Y$ is a scheme and $X$ is a separated Deligne--Mumford stack.

\begin{example}\label{counter}
	Let $p$ be a prime, $R$ a DVR whose fraction field $K$ has characteristic $0$ and contains a $p$-th root of $1$, denoted by $\zeta_{p}$, while its residue field $k$ has characteristic $p$ and is not perfect. An example would be the localization of $\ZZ[\zeta_{p}][t]$ at a prime ideal of height~1 containing $p$.
	
	Choose an element $a\in R^{*}$ whose image in $k$ is not a $p$-th power, and set $R' \eqdef R(\radice[p]{a})$. Then $R'$ is a DVR, since the $R'\otimes_{R}k = k(\radice[p]{\overline{a}})$ is a field (here $\overline{a}$ is the class of $a$ in $k$). Write $K' = K(\radice[p]{a})$ for its fraction field and $k' = k(\radice[p]{\overline{a}})$ for its residue field.
	
	Call $C_{p}$ the cyclic group of order $p$ generated by $\zeta_{p} \in K^{*}$. The extension $K'/K$ is Galois with cyclic Galois group $C_{p}$ acting by $\radice[p]{a}\mapsto\zeta_{p}\radice[p]{a}$. The action of $C_{p}$ on $K'$ naturally extends to $R'$.
	
	Let $X$ be the quotient stack $[\spec R'/C_{p}]$; this is a separated Deligne--Mumford stack, but it is not tame. Since $(R')^{C_{p}} = R$ the moduli space of $X$ is $\spec R$, and we have  we have a natural map $X \arr \spec R$, which is an isomorphism over $\spec K \subseteq \spec R$. Since $k'/k$ is purely inseparable, then $X_{k}(k)$ is empty: such a $k$-rational point would correspond to a $C_{p}$-torsor $\spec A\to\spec k$ with an equivariant morphism $\spec A\to \spec k'$ and thus an embedding of $k'$ in the étale $k$-algebra $A$, which is clearly absurd. In particular, there is no map $\radice[n]{\spec R}\to X$ for any $n$.
\end{example}

\begin{definition}
	In the situation of Theorem~\ref{thm:valuative}, we call the integer $n$ the \emph{loop index} of the morphism $\spec K\to X$ at the place associated with $R \subseteq K$. If the loop index is $1$, we say that $\spec K\to X$ is \emph{untangled}.
\end{definition}

\begin{lemma}\label{loopram}
	Let $R \subseteq R'$ be an extension of DVRs with ramification index $e$, and let $K \subseteq K'$ be the fraction fields of $R$ and $R'$ respectively. If $X$ is a tame stack proper over $R$ and $\spec K\to X$ is a morphism with loop index $n$, the composite $\spec K' \to \spec K\to X$ has loop index equal to $n/\gcd(n,e)$.
\end{lemma}

	\begin{proof}
		Write $m \eqdef n/\gcd(n,e)$; the statement follows from the fact that there is a natural representable morphism $\radice[m]{\spec R'}\to\radice[n]{\spec R}$ inducing the given morphism $\spec K' \arr \spec K$.
	\end{proof}

We spend the rest of this section proving Theorem~\ref{thm:valuative}. Given a DVR $R$ and $\pi\in R$ a uniformizing parameter, write $R^{(n)} \eqdef R[t]/(t^{n}-\pi)$, $K^{(n)} \eqdef K[t]/(t^{n}-\pi)$. We have $\radice[n]{\spec R} = [\spec R^{(n)}/\mmu_{n}]$. 

\begin{lemma}\label{rootmor}
	Let $R$ be a DVR, $m,n$ integers. A morphism $\radice[m]{\spec R}\to\radice[n]{\spec R}$ over $R$ exists if and only if $n|m$, and in this case it is unique up to equivalence.
\end{lemma}

\begin{proof}
	This follows from the fact that a section $\spec R^{(m)}\to\radice[n]{\spec R}$ exists if and only if $n|m$, and in this case it is unique up to equivalence.
\end{proof}

\begin{lemma}\label{normalroot}
	Let $R$ be a DVR, $D\subset \spec R$ the divisor corresponding to the closed point, $m,n_{i},r_{i}$ for $i=1,\dots,m$ positive integers, with $n_{i}\ge 2$ for every $i$. The fibered product
	\[X=\prod_{i=1}^{m}\radice[n_{i}]{\spec R,r_{i}D}.\]
	is normal if and only if $m=1$ and $r_{1}=1$.
\end{lemma}

\begin{proof}
	Let $V_{n,r}$ be the scheme $\spec R[t,s]_{s}/(t^{n}s-\pi^{r})$, there is an action of $\GG_{m}$ on $V_{n,r}$ given by $(\lambda,t,s) \mapsto (\lambda t,\lambda^{-n}s)$ and $\radice[n]{\spec R,rD}\simeq [V_{n,r}/\GG_{m}]$, see \cite[Appendix B]{dan-tom-angelo2008}. Consider the fibered product 
	\[Y=\prod_{i}V_{n_{i},r_{i}}=R[t_{1},s_{1},\dots,t_{m},s_{m}]/(t_{i}^{n_{i}}s_{i}-\pi^{r_{i}}).\]
	The prime ideal $\frp=(t_{1},\dots,t_{m},\pi)$ is the generic point of the special fiber and has height $1$. Since $V_{r,n}\to\radice[n]{\spec R,rD}$ is smooth, then $Y\to X$ is smooth and hence $X$ is normal if and only if $Y$ is normal at $\frp$.
	
	Now consider the prime ideal $\frp_{0}=(t_{1},\dots,t_{m},\pi)\subset R[t_{1},s_{1},\dots,t_{m},s_{m}]$, $\frp$ and $\frp_{0}$ have equal residue fields and there is a natural surjective linear map $\frp_{0}/\frp_{0}^{2}\to\frp/\frp^{2}$. We have that $\frp_{0}/\frp_{0}^{2}$ has dimension $m+1$ generated by the classes of $[t_{i}],[\pi]$. If $m=1$ and $r_{1}=1$, then $\frp_{0}/\frp_{0}^{2}$ has dimension $2$ and $[\pi]$ is in the kernel of $\frp_{0}/\frp_{0}^{2}\to\frp/\frp^{2}$, hence $Y$ is normal at $\frp$.
	
	On the other hand, assume that $Y$ is normal at $\frp$, so that $\frp/\frp^{2}$ has dimension $1$. Since $n_{i}\ge 2$ for every $i$, the kernel of $\frp_{0}/\frp_{0}^{2}\to\frp/\frp^{2}$ is generated by the classes $[\pi^{r_{i}}]$ and hence has dimension $1$ if $r_{i}=1$ for some $i$, and dimension $0$ otherwise. Since $\frp_{0}/\frp_{0}^{2}$ has dimension $m+1$ and $\frp/\frp^{2}$ has dimension $1$, this implies that $m=1$ and $r_{1}=1$.
\end{proof}

\begin{lemma}\label{easyval}
	Let $A$ be a Dedekind domain with fraction field $K$, $D\subset \spec A$ an effective, reduced divisor. Let $f:X \arr \radice[n]{(\spec A,D)}$ be a representable, proper morphism. Every generic section $\spec K\to X$ of $f$ extends uniquely to a global section $\radice[n]{(\spec A,D)}\to X$.
	\begin{proof}
		Let $Y\subset X$ be the schematic image of a generic section $\spec K\to X$, we want to prove that $Y \arr \radice[n]{(\spec A,D)}$ is an isomorphism. Since the problem is local, we may assume that $A=R$ is a DVR and $D$ is either empty or the closed point. If $D$ is empty, then $\radice[n]{\spec R,D}=\spec R$ and this is simply the valuative criterion of properness. Suppose that $D$ is the closed point. Consider the flat morphism $\spec R^{(n)} \arr \radice[n]{\spec R}$ and write $X'=X\times_{\radice[n]{\spec R}}\spec R^{(n)}$, $Y'=Y\times_{\radice[n]{\spec R}}\spec R^{(n)}$. Thanks to \cite[\href{https://stacks.math.columbia.edu/tag/0CMK}{Lemma 0CMK}]{stacks-project} we have that $Y'\subset X'$ is the schematic image of the induced generic section $\spec K^{(n)}\to X'$. By the valuative criterion of properness, there is a section $\spec R^{(n)}\to X'$ which is a closed immersion since $X'$ is representable, this implies that $Y'\to\spec R^{(n)}$ is an isomorphism. It follows that $Y \arr \radice[n]{\spec R}$ is an isomorphism, too.
	\end{proof}
\end{lemma}

\begin{lemma}
	Let $R$ be a DVR, $n,m$ positive integers. Assume that $n$ is prime with the residue characteristic of $R$. Consider the $\mu_{n}$-torsor $\spec R^{(n)}\to\radice[n]{\spec R}$. There exists a unique way of extending the action of $\mu_{n}$ to $\radice[m]{\spec R^{(n)}}$, and the quotient $[\radice[m]{\spec R^{(n)}}/\mu_{n}]$ is isomorphic to $\radice[mn]{\spec R}$.
\end{lemma}

\begin{proof}
	We have a natural action $\rho:\radice[m]{\spec R^{(n)}}\times_{R}\mu_{n}\to \radice[m]{\spec R^{(n)}}$ induced by the action on $\spec R^{(n)}$. The action $\rho$ gives a structure of $\mu_{n}$-torsor to the natural morphism $\radice[m]{\spec R^{(n)}}\to\radice[mn]{\spec R}$. Let $\eta:\radice[m]{\spec R^{(n)}}\times_{R}\mu_{n}\to \radice[m]{\spec R^{(n)}}$ be an action such that the diagram
	\[\begin{tikzcd}
		\radice[m]{\spec R^{(n)}}\times_{R}\mu_{n}\dar\rar["\eta"]	&	\radice[m]{\spec R^{(n)}}\dar	\\
		\spec R^{(n)}\times_{R}\mu_{n}\rar							&	\spec R^{(n)}
	\end{tikzcd}\]
	is $2$-commutative, we want to show that $\rho$ and $\eta$ are equivalent.
	
	Let $D\subset \spec R^{(n)}\times_{R}\mu_{n}$ be the pullback of the closed point of $\spec R^{(n)}$, since $\mu_{n}$ is finite étale over $R$ we have that $D$ is a reduced divisor. Since $\mu_{n}$ is étale, the natural morphism $\radice[m]{\spec R^{(n)}\times_{R}\mu_{n},D}\to\radice[m]{\spec R^{(n)}}\times_{R}\mu_{n}$ is an isomorphism. The scheme $\spec R^{(n)}\times_{R}\mu_{n}$ is finite étale over $\spec R^{(n)}$, hence it is a disjoint union of Dedekind domains, and $\radice[m]{\spec R^{(n)}}\times_{R}\mu_{n}=\radice[m]{\spec R^{(n)}\times_{R}\mu_{n},D}$ is a disjoint union of root stacks over Dedekind domains.
	
	The stack $\underisom(\rho,\eta)$ has a proper, representable morphism 
	\[\underisom(\rho,\eta)\to \radice[m]{\spec R^{(n)}}\times_{R}\mu_{n},\]
	and for every connected component of $\radice[m]{\spec R^{(n)}}\times_{R}\mu_{n}$ there is a generic section. By Lemma~\ref{easyval}, these generic sections extend to global sections, hence $\eta\simeq \rho$.
\end{proof}

\begin{corollary}\label{rootbc}
	Let $R$ be a DVR, $n,m$ positive integers. Assume that $n$ is prime with the residue characteristic of $R$. Let $X\to\radice[n]{\spec R}$ be a morphism, and assume that the base change of $X$ to $\spec R^{(n)}$ is isomorphic to $\radice[m]{\spec R^{(n)}}$. Then $X\simeq\radice[mn]{\spec R}$.
\end{corollary}

\begin{lemma}\label{etale}
	Let $R'/R$ be a local, quasi-finite étale extension of DVRs and $X$ a tame stack over $R$, $X'\eqdef X_{R'}$. If $X'\simeq\radice[n]{\spec R'}$, then $X\simeq\radice[n]{\spec R}$.
\end{lemma}

\begin{proof}
	Let $K$ be the residue field of $R$, clearly we have that $X_{K}\to\spec K$ is an isomorphism. Write $A=\spec R'\otimes_{R}R'$, since $R'$ is quasi-finite étale over $R$ then $A$ is a product of Dedekind domains with a finite number of closed points. Let $D\subset \spec A$ be the effective, reduced divisor of all closed points and $S\eqdef \radice[n]{\spec R'}\times_{\radice[n]{\spec R}}\radice[n]{\spec R'}$, it is easy to see that $S\simeq\radice[n]{(\spec A,D)}$.
	
	Let $\phi:\radice[n]{\spec R'}\simeq X'\to X$ be the composite, and consider the two projections $p_{1},p_{2}:S\to\radice[n]{\spec R'}$. Since $X_{R'}$ is separated, then $X$ is separated, too, and hence $\underisom(p_{1}^{*}\phi,p_{2}^{*}\phi)$ is an algebraic stack with a proper, representable morphism to $S$. There is a generic section $S_{K}\to \underisom(p_{1}^{*}\phi,p_{2}^{*}\phi)$ which extends to a global section thanks to Lemma~\ref{easyval}, this gives descent data for a morphism $f:\radice[n]{\spec R}\to X$ (the cocycle condition can be checked on the generic point, where it is obvious). Since the base change to $R'$ of $f$ is an isomorphism, we have that $f$ is an isomorphism, too.
\end{proof}

\begin{proposition}\label{tameroot}
	Let $X$ be a normal, tame stack of finite type over a DVR $R$, and assume that there is a generic section $\spec K\to X$ which is an open, scheme-theoretically dense embedding. Then $X\simeq\radice[n]{\spec R}$ for some $n$.
\end{proposition}

\begin{proof}
	Since $X$ is of finite type over $R$, there exists a DVR $R_{0}\subset R$ which is the localization of a $\ZZ$-algebra of finite type and a stack $X_{0}/R_{0}$ such that $X\simeq X_{0,R}$. Furthermore, we may assume that the uniformizing parameter of $R_{0}$ maps to a uniformizing parameter of $R$, so that $\radice[n]{\spec R_{0}}\times_{R_{0}}\spec R\simeq\radice[n]{\spec R}$. Up to replacing $R,X$ with $R_{0},X_{0}$ we may assume that $R$ is Nagata.	Let $k$ be the residue field of $R$ and $p$ its characteristic.
	
	By \cite[Theorem 3.2]{dan-olsson-vistoli1}, there exists a DVR $R'$ quasi-finite and étale over $R$ and a finite, flat, linearly reductive group scheme $G/R'$ with an action on a scheme $U$ finite over $R'$ such that $X_{R'}\simeq[U/G]$. Up to enlarging $R'$, by \cite[Lemma 2.20]{dan-olsson-vistoli1} there exists a diagonalizable flat, closed subgroup $\Delta\subset G$ such that $H\eqdef G/\Delta$ is constant and tame. We may furthermore assume that the degree of $\Delta$ is a power of $p$. Thanks to Lemma~\ref{etale}, we may assume $R'=R$.
	
	{\bf Case 1.} $X$ is tame and Deligne-Mumford. Since $\Delta_{k}$ is connected and $X$ is Deligne-Mumford and generically a scheme, the action of $\Delta$ is free (because otherwise $X$ would have ramified inertia), hence up to replacing $U$ with $U/\Delta$ we may assume that $G$ is constant and tame. Since $X$ is normal and $G$ is constant and tame, then $U$ is normal, too. If $u\in U$ is a geometric point, the stabilizer $G_{u}$ acts faithfully on the tangent space, hence the automorphism groups of the points of $X$ are cyclic and tame. By \cite[Lemma 8.5]{ryd11} and Lemma~\ref{normalroot}, since $X$ is normal we have $X\simeq\radice[n]{\spec R}$ for some $n$.
	
	{\bf Case 2.} $X$ is tame. Let $V\eqdef U/\Delta$ and $Y_{0}\eqdef [V/H]$, we have that $Y_{0}$ is Deligne-Mumford and there is a natural birational morphism $X\to Y_{0}$ whose relative inertia is diagonalizable. Let $Y\to Y_{0}$ be the normalization, it is finite over $Y_{0}$ since $R$ is Nagata and since $X$ is normal the morphism $X\to Y_{0}$ lifts to a morphism $X\to Y$. By case $1$, there exists an $n$ prime with $p$ and an isomorphism $Y\simeq\radice[n]{\spec R}$. Consider the morphism $\spec R^{(n)}\to\radice[n]{\spec R}$, it is a $\mu_{n}$-torsor and hence finite étale since $n$ is prime with $p$, it follows that the base change $X\times_{\radice[n]{\spec R}}\spec R^{(n)}$ is normal with diagonalizable inertia. By \cite[Lemma 8.5]{ryd11} and Lemma~\ref{normalroot}, we have $X\times_{\radice[n]{\spec R}}\spec R^{(n)}\simeq \radice[m]{\spec R^{(n)}}$ for some integer $m$, hence $X\simeq\radice[mn]{\spec R}$ thanks to Corollary~\ref{rootbc}. 
\end{proof}

\begin{proof}[Proof of Theorem~\ref{thm:valuative}]

By base change, we may assume that $Y=\spec R$ and that $X$ is a tame stack proper over $R$. With an argument similar to the one in the proof of Proposition~\ref{tameroot}, we may reduce to the case in which $R$ is Nagata.

By \cite[Theorem B]{ryd11}, we may assume that $\spec K\to X$ is an open, scheme theoretically dense embedding. Since $R$ is Nagata, the normalization $\overline{X}$ is finite and representable over $X$. By Proposition~\ref{tameroot} we have $\overline{X}=\radice[n]{\spec R}$, hence an extension exists. If $m$ is another integer with a representable extension $\radice[m]{\spec R}\to X$, it factors through $\overline{X}=\radice[n]{\spec R}$ since $\radice[m]{\spec R}$ is normal by Lemma~\ref{normalroot}. We conclude the proof of Theorem~\ref{thm:valuative} by Lemma~\ref{rootmor}.

\end{proof}

\section{The Lang--Nishimura theorem}\label{sec:lang-nishimura}

Here is our version of the Lang--Nishimura theorem for tame stacks.

\begin{theorem}\label{LN}
	Let $S$ be a scheme and $X\dashrightarrow Y$ a rational map of algebraic stacks over $S$, with $X$ locally noetherian and integral and $Y$ tame and proper over $S$. Let $k$ be a field, $s\colon \spec k \arr S$ a morphism. Assume that $s$ lifts to a regular point $\spec k \arr X$; then it also lifts to a morphism $\spec k \arr Y$.
\end{theorem}

In the standard version of the Lang--Nishimura theorem (see for example \cite[Theorem~3.6.11]{poonen-book}), which is a standard tool in arithmetic geometry, $X$ and $Y$ are schemes, and $S = \spec k$. In the applications that we have in mind, the additional flexibility of having a base scheme is important.
	
	\begin{proof}
		According to \cite[Théorème~6.3]{laumon-moret-bailly} we can find a smooth morphism $U \arr X$ with a lifting $\spec k \arr U$ of $\spec k \arr X$; hence we can replace $X$ by $U$, and assume that $X$ is scheme. Furthermore, if $x$ denotes the image of $\spec k \arr X$ and $k(x)$ its residue field, we have a factorization $\spec k \arr \spec k(x) \arr X$, and we may assume $k = k(x)$. If $x$ has height $0$, then $\spec k$ dominates $X$, and the composite $\spec k \arr X \dashrightarrow Y$ is well defined.
		
		Otherwise, call $U \subseteq X$ the open subset where $f$ is defined. By \cite[Lemma 4.3]{giulio-angelo-genericity} there exists a DVR $R$ with residue field $k = k(x)$ and a morphism $\spec R \arr X$ that maps the generic point $\spec K$ of $\spec R$ into $U$. Thus we get a morphism $\spec K \arr U$, and we apply Corollary~\ref{cor:valuative} to the diagram
		   \[
   \begin{tikzcd}
   \spec K \rar\dar[hook] & U \dar[hook]\rar["f"] & Y \dar\\
   \spec R \rar & X \rar & S\
   \end{tikzcd}
   \]
thus getting the desired morphism $\spec k \arr Y$.
	\end{proof}

The Lang--Nishimura theorem fails for non-tame separated stacks. Let us give two examples, one in mixed characteristic, the other in positive characteristic.

\begin{example}\label{counterLN}
Let $X \arr \spec R$ be the stack constructed in \ref{counter}, it is a non-tame regular Deligne-Mumford stack. Let $k$ be the residue field of $R$. There is a rational map $\spec R\dashrightarrow X$ and $\spec R$ has a $k$ rational point, but $X$ has no $k$-rational points.
\end{example}

\begin{example}\label{counterLNP}
	Let $C_{0}$ be a smooth, projective curve of positive genus over a finite field $F$ of characteristic $p$ with $C_{0}(F)\neq \emptyset$. Let $a$ be an indeterminate, write $k \eqdef F(a)$ and $C \eqdef C_{0,k}$; since $C_{0}$ has positive genus $C(k) = C_{0}(F)$ is finite. Let $f\in k(C)$ be a rational function such that each rational point is a zero of $f$ (this can be easily found using Riemann-Roch). Consider the ramified cover $D\to C$ given by the equation
	\[t^{p}-f^{p-1}t=a\,;\]
in other words, $D$ is the smooth projective curve associated with the field extension $k(C)[t]/(t^{p}-f^{p-1}t-a)$. Let $c\in C(k)$ be a rational point, and write $R_{c} \eqdef \cO_{C,c}[t]/(t^{p}-f^{p-1}t-a)$, it is a normal domain: if $\overline{R}_{c}$ is the normalization, both $R_{c}\otimes k(C)\to \overline{R}_{c}\otimes k(C)$ and $R_{c}\otimes k\to \overline{R}_{c}\otimes k$ are isomorphisms for degree reasons since $R_{c}\otimes k = k[t]/(t^{p}-a)$ is a field of degree $p$ over $k$. Hence, $R_{c}$ is a DVR with residue field $k' \eqdef k[t]/(t^{p}-a)$. It follows that $D$ has no $k$-rational points.
	
	The cyclic group $C_{p}$ acts on $D$ by $t\mapsto t+f$, the field extension $k(D)/k(C)$ is a cyclic Galois cover and $C$ is the quotient scheme $D/C_{p}$. Let $X$ be the quotient stack $[D/C_{p}]$, there is a natural birational morphism $X\to C=D/C_{p}$. A rational point $\spec k\to X$ corresponds to a $C_{p}$-torsor $\spec A\to\spec k$ with an equivariant morphism $\spec A\to D$: since the fibers of $D\to C$ over rational points are isomorphic to $\spec k'$, a rational point of $X$ gives an embedding of $k'$ in an étale algebra $A$, which is clearly absurd. It follows that $X$ is a proper Deligne-Mumford stack over $k$ with $X(k)=\emptyset$ and a birational map $C\dashrightarrow X$.
\end{example}

As a consequence of Theorem~\ref{LN}, we can decide whether a residual gerbe of a tame stack is neutral or not by looking at a resolution of singularities of the coarse moduli space. We find this rather surprising.

\begin{corollary}\label{cor:moduli}
	Let $X$ be a locally noetherian, regular and integral tame stack with coarse moduli space $X\to M$, and $\overline{M} \arr M$ a proper birational morphism, with $\overline{M}$ integral and regular. Assume that there is a lifting $\spec k(M) \arr X$ of the generic point\/ $\spec k(M) \arr M$.
	
	If $k$ is a field and $m\colon \spec k\to M$ a morphism, then $m$ lifts to a morphism $\spec k \arr X$ if and only if it lifts to a morphism $\spec k \arr \overline{M}$.\qed
\end{corollary}

So, for example, if $M$ is regular all morphisms $\spec k \arr M$ lift to $X$, and all residual gerbes are neutral.

\bibliographystyle{amsalpha}
\bibliography{main}

\end{document}